\documentclass{amsart}
\usepackage{amsmath}
\usepackage{amsthm}
\usepackage{graphicx}
\usepackage{amssymb}
\usepackage{epstopdf}
\usepackage{nicefrac}
\usepackage{mathrsfs}
\usepackage{mathrsfs}
\usepackage{hyperref}
\usepackage{enumitem} 
\usepackage{tikz}
\usepackage{caption}
\usepackage{subcaption}
\theoremstyle{plain}

\newtheorem{theo}{Theorem}[section]
\newtheorem{lm}[theo]{Lemma}

\newtheorem{cor}[theo]{Corollary}

\DeclareMathOperator{\dist}{dist} 
 
\title[integrability of inner functions]{Inner Functions, Möbius Distortion and Angular Derivatives}
\author{Konstantinos Bampouras}
\address{Department of Mathematical Sciences, Norwegian University of Sciences and Technology (NTNU), NO-7491 Trondheim, Norway}
\email{konstantinos.bampouras@ntnu.no}
\author{Artur Nicolau}
\address{Departament de Matem\`atiques, Universitat Aut\`onoma de Barcelona, and Centre de Recerca Matem\`atica, 08193, Barcelona, Spain}
\email{artur.nicolau@uab.cat}
\date{\today}

\usepackage{cite}
\begin{document}
	\begin{abstract}
		We prove that an inner function has finite $\mathcal{L} (p)$-entropy if and only if its accumulated Möbius distortion is in $L^p$, $0<p<\infty$. We also study the support of the positive singular measures such that their corresponding singular inner functions have finite $\mathcal{L} (p)$-entropy.  
	\end{abstract}
	
	
	\thanks{ 
		The second author is supported in part by the Generalitat de Catalunya (grant 2021 SGR 00071), the Spanish Ministerio de Ciencia e Innovaci\'on (project PID2021-123151NB-I00) and the Spanish Research Agency through the Mar\'ia de Maeztu Program (CEX2020-001084-M)}

	\maketitle

	\section{Introduction}
	
	Let $\mathbb{D}$ be the open unit disc in the complex plane and let $d_h (z, w)$ denote the  hyperbolic distance between the points $z, w \in \mathbb{D}$ given by  
	$$
	d_h (z,w) = \inf_\gamma \int_\gamma \frac{2 |d \zeta |}{1-|\zeta|^2} =   \log \frac{1+\rho(z,w)}{1- \rho (z,w)},
	$$
	where the infimum is taken over all curves $\gamma$ contained in $\mathbb{D}$ joining $z$ and $w$ and $\rho(z,w)= |z-w|/|1- \overline{w}z|$, $z, w \in \mathbb{D}$. The Schwarz lemma says that any analytic self-mapping $f$ of the unit disc is a contraction in the hyperbolic metric, that is, $d_h (f(z), f(w) ) \leq d_h (z,w)$ for any $z, w \in \mathbb{D}$, or equivalently, the hyperbolic derivative of $f$, denoted by $D_hf $, satisfies
	\begin{equation}
		\label{HYPDER}
		D_h f (z) = \frac{(1- |z|^2) |f'(z)|}{1- |f(z)|^2} \leq 1, \quad z \in \mathbb{D} . 
	\end{equation}
	Moreover, equality at a single point implies equality at every point in the unit disc and that $f$ is an automorphism of $\mathbb{D}$. 
	
	An analytic self-mapping $f$ of the unit disc is said to have a {\em finite
		angular derivative} (in the sense of Carath\'eodory) at a point $\xi \in \mathbb{T}$ if $\lim_{r \to 1} f(r\xi )$ exists and belongs to the unit circle and $ \lim_{r \to 1} f' (r \xi) $ exists finitely. In this case, we write $|f' (\xi)| = \lim_{r \to 1} |f' (r \xi)|$. It is well known that $f$ has a finite angular derivative at a point $\xi \in \mathbb{T}$ if and only if
	\begin{equation}
		\label{ANGDER}
		\liminf_{z \to \xi} \frac{1- |f(z)|}{1- |z|} < \infty. 
	\end{equation}
	If $f$ does not have a finite angular derivative at $\xi$, it is customary to set $|f'(\xi)| := \infty$. With these notations it is well known that
	$$
	|f'(\xi)| = \lim \frac{1- |f(z)|}{1-|z|}, \quad \xi \in \mathbb{T} ,  
	$$
	where the limit is taken as $z \in \mathbb{D}$ tends non-tangentially to $\xi$. See for instance \cite{garnett} or Chapter IV of \cite{shapiro}. 
	
	Let $m$ be the normalized Lebesgue measure on the unit circle $ \mathbb{T}$. Inner functions are analytic self-mappings $f$ of the unit disc such that their radial limits $\lim_{r \to 1} f(r \xi)$ have modulus one at $m$-almost every point  $\xi \in \mathbb{T}$. Fix $0 <  p < \infty$. An inner function $f$ has finite $\mathcal{L} (p)$-entropy if its angular derivative $|f'(\xi)|$ is finite at $m$-almost every point $\xi \in \mathbb{T}$ and $\log |f'(\xi)| \in L^p (\mathbb{T})$. 
	Note that when $p=1$ these are the inner functions $f$ with finite entropy, or equivalently, inner functions whose derivative is in the Nevanlinna class, which were first studied in \cite{craizer} and have recently attracted some attention \cite{ivrii19 , ivriinicolau1, ivriikreitner, ivrii2024innerfunctionslaminations}. 
	
	Let $f$ be an analytic self-mapping of the unit disc. Define the {\em Möbius distortion} of $f$ as 
	$$
	\mu (f) (z)   \ = \ 1 - D_h (f) (z), \qquad z \in \mathbb{D}.
	$$
	The Möbius distortion $\mu(f) (z)$ measures how much $f$ deviates from an automorphism of $\mathbb{D}$ near $z \in \mathbb{D}$. Several natural classes of inner functions can be described using the Möbius distortion. For instance, M. Heins proved in \cite{MR861696} that $\mu (f) (z) \to 0$ as $|z| \to 1 $ if and only if $f$ is a finite Blaschke product. Local versions of this result can be found in \cite{krausrothrusc}. In \cite{kraus} D. Kraus proved that $f$ has a finite angular derivative at almost every point of the unit circle if and only if for almost every point $\xi \in \mathbb{T}$ we have $\mu (f) (z) \to 0$ as $z$ tends non-tangentially to $\xi$. 
	
	Consider the accumulated Möbius distortion defined as 
	\begin{equation}
		\label{AREAFUNCTION}
		A(f) (\xi ) = \int_0^1 \mu(f) (r \xi) \frac{2 dr}{1-r^2} , \quad \xi \in \mathbb{T}.
	\end{equation}
	The following pointwise estimate has been recently proved in \cite{MR4887224} and \cite{ivrii2024innerfunctionslaminations} (see also \cite{ivriinicolau2}): for every analytic self-mapping $f$ of the unit disc that fixes the origin, we have 
	\begin{equation}
		\label{POINTWISE ESTIMATE}
		A(f) (\xi ) \leq \log |f'(\xi)|, \quad \xi \in \mathbb{T}  . 
	\end{equation}
	It is worth mentioning that a converse estimate of the form $\log |f'(\xi)| \leq C_1  A(f) (\xi) + C_2$  where $C_1, C_2$ are absolute constants, can not hold. This can be seen by taking $f(z) = z (z-a) / (1 - az)$ where $0<a<1$ for which $f'(1) = 2/(1-a)$ and $A(f) (1) = \log (1+a) + \log 2$. However recently, P. Gumenyuk, M. Kourou, A. Moucha and O. Roth and independently O. Ivrii and M. Urbanski, have proved that for any analytic mapping $f: \mathbb{D} \rightarrow \mathbb{D}$ and any point $\xi \in \mathbb{T}$, the angular derivative $|f'(\xi)|$ is finite if and only if $A(f) (\xi) < \infty$. See \cite{MR4887224} and Theorem B.1 of \cite{ivrii2024innerfunctionslaminations}. The main purpose of this note is to show that even if $\log |f'|$ and $A(f)$ are not pointwise comparable, they still belong to the same $L^p (\mathbb{T})$ spaces, $0< p < \infty$.


	\begin{theo}\label{main}
		Fix $0< p < \infty$. Let $f:\mathbb{D}\to\mathbb{D}$ be an inner function. Then $f$ has finite $\mathcal{L} (p)$-entropy, that is $\log |f'| \in L^p (\mathbb{T})$, if and only if $A(f) \in L^p (\mathbb{T})$.
	\end{theo}
	When $p=1$ the result follows easily from Green's Formula and the identity
	\begin{equation}
		\label{IDENTITY}
		\Delta \biggr( \log  \frac{1-|f(z)|^2}{1-|z|^2} \biggr) = \frac{4(1 - D_h (f) (z)^2)}{(1-|z|^2)^2}, \quad z \in \mathbb{D}, 
	\end{equation}
	which one can check by direct calculation. The general case $p \neq 1$ is harder and follows from the following good-$\lambda$ inequality.
	
	\begin{theo}\label{lambdaestimate}
		Given  $0<\eta<1$ and $M>2$ there exist constants $0 < \epsilon < 1$ and $\lambda_{0}>0$ such that for any $\lambda>\lambda_{0}$ and any inner function $f: \mathbb{D} \rightarrow \mathbb{D}$ with $f(0)=0$, we have that  
		$$m(\{\xi\in\mathbb{T}:\log|f'(\xi)|\geq M \lambda \text{ and } A(f)(\xi)\leq \epsilon \lambda \})\leq \eta \, m(\{\xi\in\mathbb{T}:\log|f'(\xi)|\geq \lambda \}).$$
	\end{theo}
	
	Roughly speaking, Theorem \ref{lambdaestimate} says in a quantitative way that, even though $\log |f'|$ and $A(f)$ are not pointwise comparable, the set of points $\xi \in \mathbb{T}$ where $\log |f'(\xi)|$ is large and $A(f) (\xi)$ is small, has a small measure. Our result is inspired by the classical good $\lambda$-inequalities of Fefferman and Stein which relate the size of the non-tangential maximal function of a harmonic function and the size of its Lusin area function. See \cite{feffermanstein}. These estimates have been complemented and extended to different contexts. See for instance \cite{banuelosmoore}. The proof of Theorem \ref{lambdaestimate} uses stopping time arguments, Green's Formula and the identity \eqref{IDENTITY} and it is the most technical part of the paper. It is worth mentioning that the proof also gives local versions of both Theorems \ref{main} and \ref{lambdaestimate}.

	Recall that any inner function $f$ factors as $f= B S_\mu$ where $B$ is a  Blaschke product and $S_\mu$ is a singular inner function defined as 
	\begin{equation}
		\label{singular}
		S_\mu (z) = \exp \left( - \int_{\mathbb{T}} \frac{\xi + z}{\xi - z} d \mu (\xi) \right), \quad z \in \mathbb{D}, 
	\end{equation}
	where $\mu$ is a finite positive Borel measure on $\mathbb{T}$ singular with respect to the Lebesgue measure $m$. Given a closed set $E \subset \mathbb{T}$ let $\text{dist} \left( \xi, E\right)$ denote the distance from $\xi$ to $E$. Fix $0<p<\infty$. A closed subset $E \subset \mathbb{T}$ of Lebesgue measure zero will be called a $\mathcal{C} (p)$-Beurling--Carleson set if 
	\[
	\int_{\mathbb{T}} |\log \text{dist} \left( \xi, E\right) |^p dm(\xi) < \infty.
	\]
	Note that $\mathcal{C} (1)$-Beurling--Carleson sets are the classical Beurling--Carleson sets which appear in many problems in function theory. See for instance the references in \cite{ivriinicolau1}. Our next result relates the support of a singular measure $\mu$ with the $\mathcal{L} (p)$-entropy of $S_\mu$, extending the corresponding result for $p=1$ proved in \cite{ivrii19}. See also \cite[Theorem 2]{ivriinicolau1}. 
	
	\begin{theo}\label{appl}
		Fix $0< p < \infty$. Let $\mu$ be a finite positive Borel measure on $\mathbb{T}$ which is singular with respect $m$ and let $S_\mu$ be the corresponding singular inner function defined in \eqref{singular}. Consider the following conditions: 
		\begin{enumerate}
			\item The measure $\mu$ is supported in a $\mathcal{C} (p)$-Beurling--Carleson set.
			\item $S_\mu$ has finite $\mathcal{L} (p)$-entropy, that is, $\log|S_\mu '|\in L^{p}(\mathbb{T})$.
			\item For every $0<c<1$, the integral $$\int_{\{z\in\mathbb{D}:|S_\mu (z)|<c\}} \dfrac{  | \log (1-|z|) |^{p-1}}{1-|z|}dA(z)$$ converges.	
			\item There exists $0<c<1$ such that  the integral $$\int_{\{z\in\mathbb{D}:|S_\mu (z)|<c\}} \dfrac{  | \log (1-|z|) |^{p-1}}{1-|z|}dA(z)$$ converges.	
			\item The measure $\mu$ is concentrated in a countable union of $\mathcal{C} (p)$-Beurling--Carleson sets.
		\end{enumerate}
		Then, the implications $(1)\Rightarrow (2)\Rightarrow (3)\Rightarrow (4)\Rightarrow (5)$ hold.
	\end{theo}
	Fix $0< p < \infty$. A positive finite Borel measure $\mu$ on $\mathbb{T}$ which is singular with respect to $m$ will be called $\mathcal{L} (p)$-invisible if for every non-trivial positive measure $\nu\leq\mu$, the corresponding singular inner function $S_\nu$ satisfies $\log|S'_\nu|\notin L^p(\mathbb{T})$. The case $p=1$ was already considered in \cite{ivrii19}. As a direct consequence of Theorem \ref{appl}, we have the following description of $\mathcal{L} (p)$-invisible measures which in the case $p=1$ was already proved in \cite{ivrii19}.
	
	
	\begin{cor}
		A positive finite Borel measure $\mu$ on $\mathbb{T}$ which is singular with respect to the Lebesgue measure $m$, is $\mathcal{L} (p)$-invisible if and only if $\mu(E)=0$ for any $\mathcal{C} (p)$-Beurling--Carleson set $E$.
	\end{cor}
	
	\par This paper is organised in three further sections. The next section is devoted to some auxiliary results which will be used in the proofs of Theorem \ref{main} and Theorem \ref{lambdaestimate} which are given in Section \ref{sec3}. Lastly, in Section \ref{sec4} we prove Theorem \ref{appl}.
	
	We will use the letter $C$ to denote different absolute constants whose values may change from line to line.
	
	\section{Preliminaries}\label{sec2}
	Given an analytic mapping $f: \mathbb{D} \rightarrow \mathbb{D}$ consider
	\begin{equation}
		\label{G(f)}
		G(f) (z) = \log \frac{1- |f(z)|^2}{1- |z|^2}, \quad z \in \mathbb{D}. 
	\end{equation}
	For a point $z\in\mathbb{D} \setminus \{0\}$ we will denote by $I(z)$ the subarc of $\mathbb{T}$, centered at $z/|z|$ with $m(I(z) ) = 1- |z|$.
	We start with some elementary properties of $G(f)$ which will be used in the proof of Theorem \ref{lambdaestimate}. 
	
	\begin{lm}\label{Gprop}
		Let $f: \mathbb{D} \rightarrow \mathbb{D}$ be an analytic mapping and consider $G(f)$ as defined in \eqref{G(f)}.  Then,  
		
		(a) There exists a universal constant $C>0$ such that $|G(f)(z)-G(f)(w)|\leq Cd_h(z,w)$ for any $z,w\in\mathbb{D}.$
		
		
		(b) There exists a universal constant $C>0$ such that $G(f)(w) \geq G(f) (z) - C$ for any $z,w \in \mathbb{D}$ with $I(w) \subset I(z)$. 
		
		(c) The inequality $ | \log |z| ||\nabla G(f)(z)|\leq 4$ holds for any $z \in \mathbb{D}$ with $ \frac{1}{2} <|z|<1$. 
		
	\end{lm}
	
	\begin{proof}
		
		(a) Note that 
		$$
		G(f) (z) - G(f) (w) = \log \frac{1-|f(z)|^2}{1 - |f(w)|^2}  + \log \frac{1-|w|^2}{1 - |z|^2}, \quad z, w \in  \mathbb{D} .
		$$
		Since there exists an absolute constant $C_1>0$ such that
		\begin{equation}
			\label{log}
			\left|\log\frac{1- |w|^2}{1 - |z|^2} \right| \leq C_1 d_h (z,w), \quad z, w \in \mathbb{D} , 
		\end{equation}
		Schwarz's Lemma gives that $|G(f) (z) - G(f)(w)| \leq 2 C_1 d_h(z, w) $, $z, w \in \mathbb{D}$, which gives the estimate in (a). 
		
		(b) Note that 
		$$
		G(f) (w) - G(f) (z) = \log \frac{1-|z|^2}{1 - |w|^2}  + d_h (0, f(z)) - d_h (0, f(w) ) + 2 \log \frac{1+ |f(w)|}{1 + |f(z)|}, 
		$$
		for any $z,w \in \mathbb{D}$. Observe that there exists an absolute constant $C_1 >0$ such that for any $z,w \in \mathbb{D}$ with $I(w) \subset I(z) $, we have 
		$$
		\log \frac{1-|z|^2}{1 - |w|^2} \geq d_h (z, w) -    C_1.
		$$
		The triangular inequality gives that 
		$$
		G(f) (w) - G(f) (z) \geq d_h (z,w) - C_1 - d_h (f(z), f(w)) - 2 \log 2 
		$$ 
		and Schwarz's Lemma finishes the proof.
		
		(c) Given a differentiable function $u : \mathbb{D} \rightarrow \mathbb{R}$ we use the notation $\nabla u (z) = u_x (z) + i u_y (z)$, $z \in \mathbb{D}$. A simple computation shows that $\nabla \log(1-|f(z)|^{2})=-2\overline{f'(z)}f(z) /(1-|f(z)|^{2})$, $z\in \mathbb{D}$. Thus
		$$
		\nabla G(f)(z)=2\frac{z}{1-|z|^{2}}-2\frac{\overline{f'(z)}f(z)}{1-|f(z)|^{2}} , \quad z \in \mathbb{D}.
		$$
		Since $- \log |z|\leq 1-|z|^{2}$ for any $\frac{1}{2} < |z| < 1$, applying the Schwarz-Pick inequality we get that
		$$
		|\log |z| | |\nabla G(f)(z)|\leq 2\big|z-\overline{f'(z)}f(z)\frac{1-|z|^{2}}{1-|f(z)|^{2}}\big|\leq 2+ 2D_h (f) (z) \leq 4, \quad   \frac{1}{2}< |z| <1 .$$ 
	\end{proof}
	
	For $1 \leq p < \infty$, let $\mathcal{N}_p $ be the class of analytic functions $g$ in $\mathbb{D}$ such that
	$$
	\sup_{0< r < 1} \int_{\mathbb{T}} (\log^+ |g(r \xi)|)^p dm (\xi) < \infty.  
	$$
	These classes were first considered by I. Privalov. Note that $\mathcal{N}_1$ is the Nevanlinna class. For $p>1$, Privalov proved that $g \in \mathcal{N}_p$ if and only if $g$ factors as $g=IE$, where $I$ is an inner function and $E$ is an outer function whose boundary values satisfy $\log |E| \in L^{p} (\mathbb{T})$. See \cite[pag. 93]{privalov}.
	Our next auxiliary result relates inner functions with finite $\mathcal{L} (p)$-entropy with the corresponding Privalov class.
	
	\begin{lm}\label{privalov classes}
		Let $f$ be an inner function and $1 \leq p < \infty$. Then $f$ has finite $\mathcal{L} (p)$-entropy if and only if $f' \in \mathcal{N}_p$. 
	\end{lm}
	\begin{proof}
		Since $G(f) (r \xi) \to \log |f'(\xi)|$ as $r \to 1$ for any $\xi \in \mathbb{T}$, the Schwarz-Pick inequality and part (b) of Lemma \ref{Gprop} provide an absolute constant $C>0$ such that 
		$$
		|f' (r \xi)| \leq \frac{1- |f(r \xi)|^2}{1-r^2}\leq C |f'(\xi)|,  \quad 0<r<1, \xi \in \mathbb{T} . 
		$$
		Hence if $f$ has finite $\mathcal{L} (p)$-entropy, then $f' \in \mathcal{N}_p$. Conversely assume that $f' \in \mathcal{N}_p$. Then $f'$ has radial limits at $m$-almost every point of the unit circle and thus, Fatou's lemma finishes the proof.
		
	\end{proof}

	In the proof of Theorem \ref{lambdaestimate} we will also use the following technical result.
	
	\begin{lm}\label{eqintonlines}
		There exists an absolute constant $C>0$ such that for any analytic mapping $f: \mathbb{D} \rightarrow \mathbb{D}$, for any $0< \ell < 1$ and for any pair of points $z= |z|\xi_1, w= |w| \xi_2 \in \mathbb{D}$ with $|z| > \ell $ and $|w| > \ell$, we have 
		\begin{equation}
			\label{dif}
			\big|\int_{\ell}^{|z|} \mu (f) (s\xi_1)\frac{2ds}{1-s^2}-\int_{\ell}^{|w|} \mu (f) (s\xi_2) \frac{2ds}{1-s^2}\big|\leq Cd_{h}(z, w).
		\end{equation}
	\end{lm}
	
	\begin{proof}
		Since $d_h (|z|, |w|) \leq d_h (z,w)$ for any $z,w \in \mathbb{D}$, it is sufficient to show that there exists a universal constant $C>0$ such that 
		\begin{equation}
			\label{clau}
			\big|\int_{\ell}^{|z|}D_{h}(f)(s\xi_1) \frac{2ds}{1-s^2}-\int_{\ell}^{|w|} D_{h}(f)(s\xi_2) \frac{2ds}{1-s^2}\big|\leq Cd_{h}(z, w), \, z,w \in \mathbb{D} . 
		\end{equation}
		It is well known that 
		$$
		d_h (D_h (f) (z),  D_h (f) (w)) \leq 2   d_h (z,w) , \quad  z , w \in \mathbb{D} , 
		$$
		(see \cite[Corollary 3.7]{beardonminda}). Using the elementary estimate $2x \leq \log (1+x) - \log (1-x)$ for $0 \leq x < 1$, we deduce that 
		$$
		\rho (D_h (f) (s \xi_1) ,   D_h (f) (s \xi_2) ) \leq  d_h ( s \xi_1 , s\xi_2), \quad 0 \leq s < 1. 
		$$
		Hence
		$$
		|D_h (f) (s \xi_1) -   D_h (f) (s \xi_2) | \leq  2 d_h (s \xi_1, s \xi_2),  \quad 0 \leq s < 1. 
		$$
		We first assume that $|z|= |w|$ and $d_h (z, w) \leq 1$. Then there exists a universal constant $C_1 >0$ such that $d_h (s \xi_1 , s\xi_2) \leq  C_1  |z - w| / (1-s^2)$ for any $0<s<|z|$. Hence the left hand side of \eqref{clau} is bounded by
		$$
		4 C_1 \int_{\ell}^{|z|} \frac{|z-w|}{(1- s^2)^2} ds \leq 4 C_1 \frac{|z-w|}{1- |z|} \leq C d_h (z,w) ,
		$$
		where $C>0$ is an absolute constant. 
		
		Assume now that $d_h (z, w) \leq 1$ but $|z| \neq |w|$. We can assume $|z| < |w|$. Let $I$ denote the left hand side of \eqref{clau} and consider the point $z^* = |z| \xi_2$. The previous argument shows that
		$$
		I \leq C d_h (z, z^* ) + \int_{|z|}^{|w|} D_h (f) (s \xi_2) \frac{2 ds}{1- s^2}. 
		$$
		Since the last integral is bounded by $d_h (z^* , w)$ we deduce that $I \leq C d_h (z, z^*) + d_h (z^* , w)$. Now the estimate $ d_h (z, z^*) + d_h (z^* , w) \leq C_2 d_h (z,w)$ finishes the proof in this case. 
		
		Finally assume that $d_h (z,w) >1$. Let $N$ be the positive integer satisfying $N< d_h (z,w) \leq N+1$. Pick points $z_0 = z, z_1 ,  \ldots , z_N = w$ with $d_h (z_k , z_{k+1}) \leq 1$, $k=0,1, \ldots N-1$. The previous argument gives that
		$$
		\big|\int_{\ell}^{|z_{k+1}|}D_{h}(f)(s\frac{z_{k+1}}{|z_{k+1}|}) \frac{2ds}{1-s^2}-\int_{\ell}^{|z_k|} D_{h}(f)(s \frac{z_k}{|z_k|}) \frac{2ds}{1-s^2}\big|\leq C . 
		$$
		Adding over $k=0, \ldots , N-1$ one finishes the proof.        
		
		
		
	\end{proof}

	Let $\Gamma (\xi , \alpha) = \{ z \in \mathbb{D} : |z - \xi| < \alpha (1- |z|)  \}$ be the Stolz angle with vertex at $\xi \in \mathbb{T}$ and aperture $\alpha >1$. Given an analytic self-mapping $f$ of the unit disc, consider its conical accumulated Möbius distortion defined as 
	\begin{equation}
		\label{Conic}
		B_\alpha (f) (\xi ) = \int_{\Gamma(\xi, \alpha)}\mu(f) (z) \frac{dA(z)}{(1-|z|^2)^2} , \quad \xi \in \mathbb{T}. 
	\end{equation}
	It turns out that $B_\alpha (f)$ and $A(f)$ are pointwise comparable.
	
	\begin{lm}\label{CA}
		Given $\alpha >1$ there exists a constant $C(\alpha) >0$ such that for any analytic self-mapping $f$ of the unit disc and any point $\xi \in \mathbb{T}$ one has
		$$
		C(\alpha)^{-1} A(f)(\xi ) \leq B_\alpha (f) (\xi) \leq  C(\alpha) A(f)(\xi ) . 
		$$
	\end{lm}
	
	\begin{proof}
		By \cite[Corollary 3.7]{beardonminda}, we have that $d_h(D_h(f)(z),D_h(f)(w))\leq 2d_h(z,w)$ for any pair of points $z, w \in \mathbb{D}$. Thus, for any $C_1>0$ there exists a positive constant $C_2>1$ such that  
		$$C_2^{-1} \mu (f)(w) \leq \mu(f)(z)\leq C_2 \mu (f)(w), $$
		for any pair of points $z,w\in\mathbb{D}$ with $d_h(z,w)\leq C_1$. Hence fixed $\alpha >1$, there exists a constant $C=C(\alpha) >1 $ such that
		$$
		C^{-1} \mu (f) (r \xi) \leq \mu (f) (z) \leq C \mu (f) (r \xi) 
		$$
		for any pair of points $z, r\xi \in \mathbb{D}$ with $z \in \Gamma(\xi, \alpha)$ and $(1-r)/ 2 < 1- |z| < 2(1-r)$. We deduce that 
		$$
		C^{-1} A(f)(\xi ) \leq B_\alpha (f) (\xi) \leq  C A(f)(\xi ) .
		$$

		
	\end{proof}

	\section{Proof of Theorem \ref{main}}\label{sec3}
	We start with the pointwise estimate mentioned in the Introduction  which was already proved in \cite{MR4887224} and \cite{ivrii2024innerfunctionslaminations}. For the sake of completeness we give its short proof. 
	\begin{lm}\label{pointwise}
		Let $f:\mathbb{D}\to\mathbb{D}$ be an analytic function with $f(0)=0$. Then  $A(f)(\xi)\leq \log|f'(\xi)|$ for any $\xi \in \mathbb{T}$.
	\end{lm}
	\begin{proof}
		Fix $\xi \in \mathbb{T}$. We can assume $|f'(\xi)| < \infty$ and hence $f$ has radial limit of modulus $1$ along the radius ending at $\xi$. Since
		$$
		d_h (f(R \xi), 0) \leq \int_0^R  D_h (f) (r \xi) \frac{2 dr}{1-r^2} , \quad 0 < R < 1,
		$$
		we have 
		\begin{align*}
			& \int_0^R (1 - D_h (f) (r \xi))\frac{2 dr}{1-r^2} \leq d_h (0, R \xi) - d_h (f(R \xi),  0) =  \\
			& =    \log \frac{1- |f(R \xi)|}{1-R} + \log \frac{1+ R }{1 + |f(R \xi)|} ,\quad 0<R<1 .
		\end{align*}
		The estimate follows by taking $R \to 1$. 
		
	\end{proof}
	
	We now introduce some notation. Given an arc $I \subset \mathbb{T}$ we consider the Carleson box $Q=Q(I)=\{ r \xi : \xi \in I , 0<1-r \leq m(I)\}$. If $Q=Q(I)$ is a Carleson box, it is costumary to denote $\ell(Q) = m(I)$. Given a Carleson box $Q$ denote by $I(Q) = \{z/ |z| : z \in Q\}$ its radial projection on the unit circle  and $T(Q) = \{z \in Q :  1-|z| = \ell (Q) \}$ its top side. Note that $I(Q(I))= I$. Also $\xi (I)$ denotes the center of the arc $I$ and $z(Q) = (1- \ell (Q)) \xi (I(Q))$ the center of $T(Q)$. The dyadic decomposition of an arc $I \subset \mathbb{T}$ into dyadic subarcs $\{I_{j,k} \}$ gives the corresponding dyadic decomposition of the Carleson box $Q(I)$ into dyadic subboxes $\{Q(I_{j,k}) \}$. 
	
	Given a Carleson box $Q$ we consider the local accumulated Möbius distortion defined as 
	\begin{equation}
		\label{LocalAREAFUNCTION}
		A_Q (f) (\xi ) = \int_{1- \ell (Q)}^1 \mu (f) (r \xi) \frac{2 dr}{1-r^2} , \quad \xi \in I(Q) . 
	\end{equation}

	
	The next auxiliary result is the main step in the proof of Theorem \ref{lambdaestimate}. 
	
	\begin{lm}\label{lengthestimate}
		There exists a universal constant $C>0$ such that the following statement holds. Let $f$ be an inner function with $f(0)=0$. Let $a,b,c>0$ be positive constants  and let $Q\subset\{z\in\mathbb{C}:\frac{1}{2}<|z|<1\}$ be a Carleson box such that $G(f)(z (Q)) =  a$. Consider the set 
		$$E(b,c)=\{\xi\in I(Q): \log|f'(\xi)|\geq b \text{ and } A_Q (f)(\xi)\leq c\}. $$ Then 
		$$(b-a)m(E(b,c))\leq C(c+1)m(I(Q)).$$
	\end{lm}
	\begin{proof}
		Given a Carleson box $R \subset Q$ we denote by $L(R) = [0, z(R)]  \cap Q$ the piece of the radial segment joining the origin and the point $z(R)$, contained in $Q$. By Lemma \ref{eqintonlines} there exists a universal constant $C_0 >0$ such that for any Carleson box $R \subset Q$ and any point $w \in T(R)$ we have
		$$
		\left|   \int_{L(R)}  \mu (f) (z) \frac{2 |dz|}{1- |z|^2}   -    \int_{[0, w] \cap Q }  \mu (f) (z) \frac{2 |dz|}{1- |z|^2}    \right| \leq C_0 .
		$$
		Now we use the following stopping time argument. Let $\{R_{j} \}$ be the collection of maximal dyadic Carleson sub-boxes of $Q$ such that
		$$
		\int_{L(R_j)} \mu (f) (z) \frac{2 |dz|}{1- |z|^2} > c + C_0 .
		$$
		Note that the maximality and Lemma \ref{eqintonlines} give that there exists a universal constant $C>0$ such that 
		\begin{equation}
			\label{max}
			c + C_0 < \int_{L(R_j)} \mu (f) (z) \frac{2 |dz|}{1- |z|^2} < c + C_0 +  C . 
		\end{equation}
		The choice of $C_0$ gives that 
		\begin{equation}
			\label{C0}
			\int_{[0, w] \cap Q }  \mu (f) (z) \frac{2 |dz|}{1- |z|^2} >c, 
		\end{equation}
		for any $w \in T(R_j)$ and any $R_j \in \{R_j \}$.
		Let us define $B=Q\setminus \cup_{j}R_{j}$ and $B_n=B\cap\{z\in\mathbb{D}:|z|\leq  r_n \}$ where $r_n = 1-2^{-n} \ell (Q)$. Applying Green's formula to the functions $G(f)$ defined in \eqref{G(f)} and $v (z)= - \log |z|$, we have that
		\begin{equation}
			\label{green}
			\int_{\partial B_n} \left(v (z)\frac{\partial G(f)(z)}{\partial n}-G(f)(z)\frac{\partial v (z)}{\partial n}  \right)d\sigma (z) = \int_{B_n} v (z)\Delta(G(f)(z))dA(z) , 
		\end{equation}
		where $dA$ denotes the area measure and $d \sigma$ denotes the linear measure on $\partial B_n$. Let us name these integrals as
		$$I_1=\int_{\partial B_n} v (z)\frac{\partial G(f)(z)}{\partial n}d\sigma (z),$$
		$$I_2=\int_{\partial B_n}G(f)(z)\frac{\partial v (z)}{\partial n}d\sigma (z) ,$$
		and 
		$$I_3=\int_{B_n} v (z)\Delta(G(f)(z))dA(z).$$
		First we estimate $I_1$. By Lemma \ref{Gprop}, part (c), we have
		\begin{equation}
			\label{I_1}
			|I_1|=|\int_{\partial{B}}v (z)\frac{\partial G(f)(z)}{\partial n}d\sigma|\leq 4 \sigma (\partial B)\leq C m(I(Q)).
		\end{equation}
		For $I_2$, note that one can decompose $\partial B_n = C_n \cup D_n $ with 
		$$C_n = T(Q) \cup (\cup_{\mathcal{A}_n} T(R_j)) \cup J_n , $$ where $\mathcal{A}_n$ is the subfamily of those $R_j$ with $\ell (R_j) \geq 1 - r_n$ and $J_n \subset \{z \in \mathbb{D} : |z| = r_n \}$, while $D_n$ is contained in a finite union of radius emanating from the origin. Roughly speaking, $\partial B_n$ is decomposed in a circular part $C_n$ and a radial part $D_n$. Since $v$ is radial, its normal derivative vanishes on $D_n$. On the other hand  
		$$
		\frac{\partial v (z)}{\partial n} = \frac{-1}{|z|}, z \in T(Q) ; \quad  \frac{\partial v(z)}{\partial n} = \frac{1}{|z|}, z  \in C_n \setminus T(Q) . 
		$$
		Hence 
		\begin{equation}
			\label{I2}
			I_2=\sum_{\mathcal{A}_n}    \frac{1}{|z(R_j)|}  \int_{T(R_j)}G(f) d\sigma + \frac{1}{r_n} \int_{J_n }G(f) d\sigma   - \frac{1}{|z(Q)|} \int_{T(Q)}G(f) d\sigma.
		\end{equation}
		Since  $G(f) (z_Q) = a$ and $m(I(Q)) = \sigma (T(Q)) / |z(Q)|$ , part (a) of Lemma \ref{Gprop} gives that there exists an absolute constant $C>0 $ such that
		\begin{equation}
			\label{a}
			\left| \frac{1}{|z(Q)|} \int_{T(Q)}G(f) d\sigma - a m(I(Q)) \right| \leq C m(I(Q))
		\end{equation}
		Since 
		$$
		m(I(Q)) = \sum_{\mathcal{A}_n } \ell (R_j) + m ( r_n^{-1} J_n ), 
		$$
		from \eqref{I2} and \eqref{a}, we deduce that
		\begin{equation}
			\label{I2nova}
			I_2 = \sum_{\mathcal{A}_n}    \frac{1}{|z(R_j)|}  \int_{T(R_j)} (G(f) -a) d\sigma + \frac{1}{r_n} \int_{J_n } (G(f) - a) d\sigma + K, 
		\end{equation}
		where $|K| \leq C m(I(Q))$. 
		
		Let us now turn our attention to $I_3$. By identity \eqref{IDENTITY}, we have
		$$
		I_3=\int_{B_n}v (z)\Delta(G(f)(z))dA(z) \leq 8 \int_{B_n} |\log |z| | \mu (f) (z) \frac{dA(z)}{(1- |z^2|)^2}. 
		$$
		Let $L(\xi)=[0,\xi]\cap B_n$. Now Fubini's theorem, estimate \eqref{max} and the choice of $C_0$ give that 
		\begin{equation}
			\label{I3}
			I_3 \leq C \int_{I(Q)} \int_{L(\xi)} \mu (f)(r\xi)\frac{2dr}{1-r^2} d\xi  \leq C(C + 2 C_0 + c) m(I(Q)) .
		\end{equation}
		Combining the above estimates \eqref{I_1}, \eqref{I2nova}, \eqref{I3}, from Green's formula in \eqref{green} we derive that there exists a universal constant $C>0$ such that
		$$\frac{1}{r_n}\int_{J_n} (G(f) - a)d\sigma+\sum_{ \mathcal{A}_n}\frac{1}{|z(R_j)|}\int_{T(R_j)}( G(f) - a) d\sigma\leq  C(1+c) m(I(Q)).$$
		By part (b) of Lemma \ref{Gprop}, there exists a universal constant $C_1>0$ such that  $G(f)(z)-a+C_1 \geq 0$ for any  $z  \in \cup_j R_j \cup J_n$. We deduce that there exist a universal constant $C>0$ such that 
		\begin{equation}
			\label{final}
			\frac{1}{r_n}\int_{J_n}(G(f)-a)d\sigma\leq C(c+1)m(I(Q)).  \end{equation}
		Consider $E_n(b,c)=\{z \in Q: |z|= r_n ,  G(f)(z)\geq b\text{ and } A_Q (f)(z / |z|)\leq c\}.$ 		Note that by \eqref{C0} we have $J_n \supset E_n(b,c)$.
		Thus, by \eqref{final}
		\begin{equation*}
			C(1+c)m(I(Q)) \geq \int_{E_n(b,c)}(G(f)-a)d\sigma \geq (b-a)m(E_n(b,c)). 
		\end{equation*}
		Therefore it suffices to notice that $\limsup m(E_n(b,c)) \geq m(E(b,c))$ which follows from the observation that $E(b,c)\subset \liminf r_n^{-1}E_n(b,c)$. 
	\end{proof}
	
	We are now ready to prove Theorem \ref{lambdaestimate}. 
	\newline
	
	\textbf{Proof of Theorem \ref{lambdaestimate}.}
	Let $C_0 >0$ be the maximum of the universal constants $C$ appearing in Lemma \ref{Gprop}, and let $\lambda_0 > 3C_0$ be a constant to be fixed later. Fix $\lambda > \lambda_0$. Let $\mathcal{A} = \{Q_j \}$ be the collection of maximal dyadic Carleson boxes $Q$ such that
	$$
	\sup \{G (f) (z) : z \in T(Q) \} > \lambda + \lambda_0 .
	$$
	Since $G(f) (0) = 0$, part (a) of Lemma \ref{Gprop} gives that $\ell (Q_j ) < 1/2$ for any $Q_j \in \mathcal{A}$. 
	Fix $Q_j \in \mathcal{A}$ and let $Q_j^*$ be the dyadic Carleson box which contains $Q_j$ with $\ell (Q_j^*) = 2 \ell (Q_j)$. The maximality gives that $\sup \{G (f) (z) : z \in T(Q_j^*) \} < \lambda + \lambda_0 .$ Since part (a) of Lemma \ref{Gprop} gives that $|G(f)(z) - G(f) (w)| \leq 2C_0$ for any pair of points $z \in T(Q_j)$, $w \in T(Q_j^*)$, we deduce that
	\begin{equation}
		\label{centre}
		\lambda \leq     G(f) (z(Q_j)) \leq  \lambda + \lambda_0 + 2 C_0.
	\end{equation}
	By part (b) of Lemma \ref{Gprop} we have
	\begin{equation}
		\label{contain}
		{\cup}_j I(Q_j) \subset \{\xi  \in  \mathbb{T}: \log |f' (\xi)| > \lambda  \}.  
	\end{equation}
	By construction we have $\log |f' (\xi)| \leq \lambda + \lambda_0$ for any $\xi \in \mathbb{T} \setminus \cup I(Q_j)$. Fix $M >2$. Since $\lambda > \lambda_0$ we have  
	$$
	\{\xi \in \mathbb{T}:  \log |f' (\xi)| > M \lambda  \} \subset \cup I(Q_j) . 
	$$
	Fix $Q_j \in \mathcal{A}$. Apply Lemma \ref{lengthestimate} with $a= G(f) (z(Q_j))$, $b= M \lambda$ and $c= \epsilon \lambda$ to obtain
	$$
	m( \{\xi \in I(Q_j) : \log |f'(\xi)| > M \lambda, A(f) (\xi) < \epsilon \lambda  \}   ) \leq C \frac{\epsilon \lambda + 1}{M \lambda - G(f) (z(Q_j))} \ell (Q_j), 
	$$
	where $C>0$ is an absolute constant. By \eqref{centre} we have
	$$
	\frac{\epsilon \lambda + 1}{M \lambda - G(f) (z(Q_j))} \leq \frac{\epsilon \lambda + 1}{(M-1) \lambda - \lambda_0 - 2 C_0} \leq \frac{\epsilon  + 1/ \lambda_0 }{(M- 2)  - 2 C_0 / \lambda_0}. 
	$$
	Given $0< \eta < 1$, taking $0 < \epsilon <1$ sufficiently small and $\lambda_0 >0$ sufficiently large, we deduce
	$$
	C \frac{\epsilon  + 1/ \lambda_0 }{(M- 2)  - 2 C_0/ \lambda_0} < \eta . 
	$$
	Hence 
	$$
	m( \{\xi \in I(Q_j) : \log |f'(\xi)| > M \lambda, A(f) (\xi) < \epsilon \lambda  \}   ) \leq \eta \, \ell (Q_j).  
	$$
	Summing over $j=1,2,\ldots$ and applying \eqref{contain}, the proof is completed. 
	\qed


	\vspace{0.1cm}
	Theorem \ref{main} follows from Theorem \ref{lambdaestimate} by standard methods.
	\vspace{0.1cm}
	
	\textbf{Proof of Theorem \ref{main}.}
	Composing $f$ with an automorphism of the unit disc, if necessary, we can assume that $f(0)=0$. By Lemma \ref{pointwise} we only need to show that $\log |f'| \in L^p (\mathbb{T})$ if $A(f) \in L^p (\mathbb{T})$. Since $f(0)=0$, Schwarz lemma gives that $|f'(\xi)| \geq 1$ for any $\xi \in  \mathbb{T}$. We use the notation $E(\lambda) = \{\xi\in\mathbb{T}:\log|f'(\xi)|\geq \lambda  \}$ for $\lambda >0$. Fix $M>2$ and apply Theorem \ref{lambdaestimate} with $\eta = 2^{-1} M^{-p}$ to find constants $0< \epsilon < 1$ and $\lambda_0 > 1$ such that
	$$
	m(\{\xi\in E(M \lambda ) : A(f)(\xi)\leq \epsilon \lambda \})\leq \frac{1}{2 M^p} \, m (E(\lambda)), 
	$$
	for any $\lambda > \lambda_0$. Then  
	\begin{align*} 
		& \int_{E(M \lambda_0)} (\log|f'|)^{p} = pM^{p}\int_{\lambda_0}^{\infty} {\lambda}^{p-1} m(E(M \lambda))d  \lambda  \\
		&\leq \frac{p}{2} \int_{\lambda_0}^{\infty} {\lambda}^{p-1} m(E(\lambda))d \lambda + p M^p \int_{\lambda_0}^{\infty} {\lambda}^{p-1} m(\{\xi\in\mathbb{T}:A(f)(\xi)\geq \epsilon \lambda \})d \lambda   \\
		& \leq \frac{1}{2} \|\log|f'|\|^{p}_{L^{p}}+\left(\frac{M}{\epsilon} \right)^{p}\|A(f)\|_{L^{p}}^p .\\ 
	\end{align*}
	Hence 
	$$
	\|\log|f'|\|^{p}_{L^{p}} \leq \frac{1}{2} \|\log|f'|\|^{p}_{L^{p}} + M^p {\epsilon}^{-p} \|A(f)\|_{L^{p}}^p + (M \lambda_0 )^p , 
	$$
	which finishes the proof. 
	\qed

	Note that the previous proof shows that given $0<p<\infty$ there exist constants $C_1,C_2>0$ depending on $p$, such that for any analytic mapping $f$ of the unit disc with $f(0)=0$, we have  
	$$\|A(f)\|^{p}_{L^{p}}\leq \|\log|f'|\|^{p}_{L^{p}}\leq C_1\|A(f)\|^{p}_{L^{p}}+C_2.$$
	The authors do not know if one can take $C_2=0$ in the estimate above. 
	

	\section{Singular Inner Functions and $\mathcal{C} (p)$-Beurling--Carleson sets}\label{sec4}
	
	We will use the following auxiliary result from \cite[Lemma 3.1]{ivriinicolau1}.
	
	\begin{lm} \label{auxili}
		Let $E \subset \mathbb{T}$ be a closed set of Lebesgue measure zero. Denote by $\{I_j\}$ the collection of its complementary arcs, that is, $\mathbb{T} \setminus E = \cup I_j $. Fix $0<p<\infty$. The following conditions are equivalent: 
		
		(a) $E$ is a $\mathcal{C} (p)$-Beurling--Carleson set.
		
		(b) $\sum |I_j| |\log |I_j||^p < \infty$. 
		
		(c) $\sum |I| | \log |I| |^{p-1}<\infty$, where the sum is taken over all dyadic arcs of $\mathbb{T}$ that meet $E$.
	\end{lm}
	
	We are now ready to prove Theorem \ref{appl}.

	\textbf{Proof of Theorem \ref{appl}.}
	\\ $(1)\Rightarrow (2)$ Assume $\mu$ is supported in a $\mathcal{C} (p)$-Beurling--Carleson set $E \subset \mathbb{T}$. Since
	$$
	\left| \frac{S_\mu '(z)}{S_\mu (z)} \right| = \left| \int_{\mathbb{T}}\dfrac{2 \xi}{(\xi-z)^{2}}d\mu(\xi) \right|
	\leq 2\mu(\mathbb{T})\dist(z, E)^{-2},  z \in \mathbb{T} \setminus  E , 
	$$ 
	it suffices to show that the integral
	$$\int_{\mathbb{T}\setminus E}|\log\dist(z,E)|^{p}d m(z),$$ converges. Consider the complementary arcs $\{I_j \}$ of $E$, that is, $\mathbb{T}\setminus E =\cup_{j}I_j$ to write the above integral as
	$$
	\sum_j\int_{I_j} |\log\dist(z,E) |^{p}d m(z).
	$$
	Since there exists a constant $C>0$ such that 
	$$
	\int_{I_j} | \log\dist(z,E) |^{p}d m(z) \leq C |I_j| | \log |I_j| |^p , \quad j=1,2,\ldots ,
	$$
	Lemma \ref{auxili} finishes the proof.
	
	
	$(2)\Rightarrow (3)$ This implication holds for any inner function $f$. Recall that Ahern proved the estimate $|f'(r \xi)|\leq 4|f' (\xi)|$ for any $0 \leq r<1$, any $\xi \in \mathbb{T}$ and any inner function $f$. See Lemma 6.1 of \cite{ahern} or \cite{mashregui}. Then
	$$
	1 - |f (r \xi)| \leq \int_r^1 |f' (s \xi)| ds  \leq 4 |f' (\xi)| (1-r), \quad 0<r<1, \xi \in \mathbb{T} . 
	$$
	Fix $0<c<1$. Using the notation $$I=\int_{\{z\in\mathbb{D}:|f (z)|<c\}}\dfrac{|\log (1-|z|)|^{p-1}}{1-|z|}dA(z), $$ we have that 
	\begin{eqnarray}
		I&=& \int_{0}^{1}\dfrac{r|\log(1-r) |^{p-1}}{1-r}m(\{\xi\in \mathbb{T}:| f (r\xi)|<c\})dr \nonumber \\
		&\leq& \int_{0}^{1}\dfrac{|\log (1-r) |^{p-1}}{1-r}m(\{\xi\in \mathbb{T}:\log|f' ( \xi)|>\log\frac{1-c}{4(1-r)}\})dr. \nonumber 
	\end{eqnarray}
	Using the change of variables $x= - \log (1-r)$, we obtain 
	$$
	I\leq \int_{0}^{\infty} x^{p-1} m(\{ \xi \in \mathbb{T}: \log|f' ( \xi )|>x + \log (1-c) - 2 \log 2 \})dx  
	$$
	which is finite whenever $\log|f'|\in L^{p}(\mathbb{T})$. 
	\\ $(3)\Rightarrow (4)$ This implication is obvious. 
	\\ $(4)\Rightarrow (5)$
	In order to prove this implication we will make use of the heavy-light decomposition from \cite[Section 4.1]{ivriinicolau1} which we now briefly recall. Let $\mu$ be a positive finite Borel singular measure on $\mathbb{T}$ and fix $M>0$. Let $ \{ I_j^{(1)} \}$ be the family of maximal dyadic arcs of $\mathbb{T}$ such that $$\frac{\mu(I_j^{(1)})}{m(I_j^{(1)})}\geq M.$$ In each arc $I_j^{(1)}$, consider the family $\{J_k^{(1)} \}$ of maximal dyadic subarcs of $I_j^{(1)}$ such that $$\frac{\mu(J_k^{(1)})}{m(J_k^{(1)})}\leq \frac{M}{100}.$$ In each $J_k^{(1)}$ we again find the maximal dyadic arcs $\{I_j^{(2)} \}$ contained in $J_k^{(1)}$ such that $$\frac{\mu(I_j^{(2)})}{m(I_j^{(2)})}\geq M.$$ Continuing this construction we are left with two families of arcs $\{I_j^{(l)} \}$ and $\{J_k^{(l)} \}$ which satisfy the following properties:
	\begin{enumerate}
		\item For any $j,l$ we have that
		$$\sum_{k:J_k^{(l)}\subset I_j^{(l)}}m(J_k^{(l)})=m(I_j^{(l)}).$$  
		\item For any $j,l$ we have that
		$$\sum_{k:I_k^{(l+1)}\subset J_j^{(l)}}m(I_k^{(l+1)})\leq \frac{1}{M}\mu(J_j^{(l)})\leq\frac{m(J_j^{(l)})}{100}.$$
		\item The measure $\mu$ is concentrated on
		$$\bigcup_{j,l} \left( \overline{I_j^{(l)} }\setminus \bigcup_{k : J_k^{(l)}\subset I_j^{(l)}} (J_k^{(l)})^{\circ} \right),$$ where $(J_k^{(l)})^{\circ}$ denotes the interior of $J_k^{(l)}$
	\end{enumerate}
	The arcs $J_k^{(l)}$ are called light and $I_j^{(l)}$ are called heavy arcs of the measure $\mu$. Now let us return to the proof of $(4)\Rightarrow (5)$.
	Let $\{I_{j}^{(i)} \}, \{ J_{k}^{(l)} \}$ be the heavy-light decomposition of $\mu$. Then it suffices to show that $E_{j,l}=\overline{I}_{j}^{(l)}\setminus \cup_{k} J_{k}^{(l)}$ is a $\mathcal{C} (p)$-Beurling--Carleson set for any $j,l$. Fix $E_{j,l}$. Let $\mathcal{B} = \mathcal{B} (j,l)$ be the collection of dyadic arcs $I$ of $\mathbb{T}$ with $m(I) < m(I_j^l)$ that meet $E_{j,l}$. Let $I \in \mathcal{B}$ and write $T_{I}=\{z\in\mathbb{D}: z / |z| \in I,\frac{|I|}{2}<1-|z|<|I|\}$. Then by construction, $\mu(I)>M  m(I) / 100$. Hence there exists an absolute constant $C>0$ such that $P [\mu] (z)\geq CM$ for all $z\in T_{I}$. Here $P[\mu]$ denotes the Poisson integral of the measure $\mu$. Since there exists an absolute constant $C_1 >0$ such that 
	$$\int_{T_{I}}\dfrac{|\log (1-|z|)|^{p-1}}{1-|z|}dA(z)\geq C_1 |I| |\log|I| |^{p-1},$$ we get that 
	$$ C_1 \sum_{I \in \mathcal{B}} |I| |\log|I||^{p-1}\leq \int_{\{z\in\mathbb{D}:P [\mu ] (z)\geq CM\}}\dfrac{|\log (1-|z|)|^{p-1}}{1-|z|}dA(z).$$ Choosing an appropriate constant $M$, Lemma \ref{auxili} finishes the proof .\qed
	
	\bibliographystyle{plain}
	\bibliography{Nicolau_Bampouras}

\begin{thebibliography}{10}

\bibitem{ahern}
Patrick Ahern.
\newblock The mean modulus and the derivative of an inner function.
\newblock {\em Indiana Univ. Math. J.}, 28(2):311--347, 1979.

\bibitem{banuelosmoore}
R.~Ba\~nuelos and C.~N. Moore.
\newblock {\em Probabilistic behavior of harmonic functions}, volume 175 of
  {\em Progress in Mathematics}.
\newblock Birkh\"auser Verlag, Basel, 1999.

\bibitem{beardonminda}
A.~F. Beardon and D.~Minda.
\newblock A multi-point {S}chwarz-{P}ick lemma.
\newblock {\em J. Anal. Math.}, 92:81--104, 2004.

\bibitem{craizer}
M.~Craizer.
\newblock Entropy of inner functions.
\newblock {\em Israel J. Math.}, 74(2-3):129--168, 1991.

\bibitem{feffermanstein}
C.~Fefferman and E.~M. Stein.
\newblock {$H\sp{p}$} spaces of several variables.
\newblock {\em Acta Math.}, 129(3-4):137--193, 1972.

\bibitem{garnett}
J.~B. Garnett.
\newblock {\em Bounded analytic functions}, volume 236 of {\em Graduate Texts
  in Mathematics}.
\newblock Springer, New York, first edition, 2007.

\bibitem{MR4887224}
Pavel Gumenyuk, Maria Kourou, Annika Moucha, and Oliver Roth.
\newblock Hyperbolic distortion and conformality at the boundary.
\newblock {\em Adv. Math.}, 470:Paper No. 110251, 52, 2025.

\bibitem{MR861696}
Maurice Heins.
\newblock Some characterizations of finite {B}laschke products of positive
  degree.
\newblock {\em J. Anal. Math.}, 46:162--166, 1986.

\bibitem{ivrii19}
O.~Ivrii.
\newblock Prescribing inner parts of derivatives of inner functions.
\newblock {\em J. Anal. Math.}, 139(2):495--519, 2019.

\bibitem{ivriikreitner}
O.~Ivrii and U.~Kreitner.
\newblock Critical values of inner functions.
\newblock {\em Adv. Math.}, 452:Paper No. 109815, 43, 2024.

\bibitem{ivriinicolau2}
O.~Ivrii and A.~Nicolau.
\newblock Analytic mappings of the unit disk which almost preserve hyperbolic
  area.
\newblock {\em Proc. London Math. Soc.}, 129(5):e70001, 2024.

\bibitem{ivriinicolau1}
O.~Ivrii and A.~Nicolau.
\newblock Beurling-{C}arleson sets, inner functions and a semilinear equation.
\newblock {\em Anal. PDE}, 17(7):2585--2618, 2024.

\bibitem{ivrii2024innerfunctionslaminations}
O.~Ivrii and M.~Urbański.
\newblock Inner functions and laminations, 2024.
\newblock arXiv:2409.11963.

\bibitem{kraus}
D.~Kraus.
\newblock Critical sets of bounded analytic functions, zero sets of {B}ergman
  spaces and nonpositive curvature.
\newblock {\em Proc. Lond. Math. Soc. (3)}, 106(4):931--956, 2013.

\bibitem{krausrothrusc}
D.~Kraus, O.~Roth, and S.~Ruscheweyh.
\newblock A boundary version of {A}hlfors' lemma, locally complete conformal
  metrics and conformally invariant reflection principles for analytic maps.
\newblock {\em J. Anal. Math.}, 101:219--256, 2007.

\bibitem{mashregui}
J.~Mashreghi.
\newblock {\em Derivatives of inner functions}, volume~31 of {\em Fields
  Institute Monographs}.
\newblock Springer, New York; Fields Institute for Research in Mathematical
  Sciences, Toronto, ON, 2013.

\bibitem{privalov}
I.~I. Privalov.
\newblock {\em Grani\v cnye svo\u istva analiti\v ceskih funkci\u i}.
\newblock Gosudarstv. Izdat. Tehn.-Teor. Lit., Moscow-Leningrad, 1950.
\newblock 2d ed.].

\bibitem{shapiro}
J.~H. Shapiro.
\newblock {\em Composition operators and classical function theory}.
\newblock Universitext: Tracts in Mathematics. Springer-Verlag, New York, 1993.

\end{thebibliography}
	
\end{document}